\documentclass[12pt,reqno]{amsart}
\usepackage{amsthm, amscd, amsfonts, amssymb, graphicx, color}
\usepackage[bookmarksnumbered, colorlinks, plainpages,linkcolor=green,anchorcolor=blue,citecolor=blue,urlcolor=blue]{hyperref}
\usepackage{enumerate}
\usepackage{exscale}
\usepackage[all]{xy}
\usepackage{relsize}
\usepackage{pdfsync}
\numberwithin{equation}{section}

\makeatletter
\@namedef{subjclassname@2020}{%
  \textup{2020} Mathematics Subject Classification}
\makeatother

\usepackage[includemp,body={398pt,550pt},footskip=30pt,%
            marginparwidth=60pt,marginparsep=10pt]{geometry}

\newtheorem{theorem}{Theorem}[section]
\newtheorem{lemma}[theorem]{Lemma}
\newtheorem{proposition}[theorem]{Proposition}

\theoremstyle{definition}

\newtheorem{remark}[theorem]{Remark}
\numberwithin{equation}{section}

\newcommand{\h}{\mathcal{H}}

\newcommand{\C}{\mathbb{C}}

\newcommand{\Z}{\mathbb{Z}}
\newcommand{\res}{\mathrm{restricted}}

\begin{document}
\title[Berezin-type operators]{A class of Berezin-type operators on weighted Fock spaces with $A_{\infty}$-type weights}
\date{September 2, 2024.
}
\author[J. Chen]{Jiale Chen}
\address{Jiale Chen, School of Mathematics and Statistics, Shaanxi Normal University, Xi'an 710119, China.}
\email{jialechen@snnu.edu.cn}

\thanks{This work was supported by the Fundamental Research Funds for the Central Universities (No. GK202207018) of China.}



\begin{abstract}
  \noindent Let $0<\alpha,\beta,t<\infty$ and $\mu$ be a positive Borel measure on $\mathbb{C}^n$. We consider the Berezin-type operator $S^{t,\alpha,\beta}_{\mu}$ defined by
  $$S^{t,\alpha,\beta}_{\mu}f(z):=\left(\int_{\mathbb{C}^n}e^{-\frac{\beta}{2}|z-u|^2}|f(u)|^te^{-\frac{\alpha t}{2}|u|^2}d\mu(u)\right)^{1/t}
  ,\quad z\in\mathbb{C}^n.$$
  We completely characterize the boundedness and compactness of $S^{t,\alpha,\beta}_{\mu}$ from the weighted Fock space $F^p_{\alpha,w}$ into the Lebesgue space $L^q(wdv)$ for all possible indices, where $w$ is a weight on $\mathbb{C}^n$ that satisfies an $A_{\infty}$-type condition. This solves an open problem raised by Zhou, Zhao and Tang [Banach J. Math. Anal. 18 (2024), Paper No. 20]. As an application, we obtain the description of the boundedness and compactness of Toeplitz-type operators acting between weighted Fock spaces induced by $A_{\infty}$-type weights.
\end{abstract}
\maketitle


\section{Introduction}
\allowdisplaybreaks[4]

The purpose of this paper is to investigate a class of Berezin-type operators acting on weighted Fock spaces. Let us begin from the definition of the spaces to work on. Given a measurable function $w$ on the $n$ dimensional complex Euclidean space $\C^n$, we say $w$ to be a weight if it is non-negative and locally integrable on $\C^n$. Given $0<p,\alpha<\infty$ and a weight $w$ on $\C^n$, the weighted space $L^p_{\alpha,w}$ consists of measurable functions $f$ on $\C^n$ such that
$$\|f\|^p_{L^p_{\alpha,w}}:=\int_{\C^n}|f(z)|^pe^{-\frac{p\alpha}{2}|z|^2}w(z)dv(z)<\infty,$$
where $dv$ is the Lebesgue measure on $\C^n$. We also write $L^p(wdv):=L^p_{0,w}$. For $0<p,\alpha<\infty$, the weighted Fock space $F^p_{\alpha,w}$ is defined by
$$F^p_{\alpha,w}:=L^p_{\alpha,w}\cap\mathcal{H}(\C^n)$$
with inherited (quasi-)norm, where $\mathcal{H}(\C^n)$ is the set of entire functions on $\C^n$. If $w\equiv1$, then we obtain the standard Fock spaces $F^p_{\alpha}$. We refer to \cite{Zh} for a brief account on Fock spaces.

Given $r>0$ and $z\in\C^n$, we use $Q_r(z)$ to denote the cube in $\C^n$ centered at $z$ with side length $r$. As usual, for $1<p<\infty$, $p'$ denotes the conjugate exponent of $p$, i.e. $1/p+1/p'=1$. Given $1<p<\infty$, a weight $w$ is said to belong to the class $A^{\res}_p$ if $w>0$ a.e. on $\C^n$ and for some (or any) fixed $r>0$,
$$\sup_{z\in\C^n}\left(\frac{1}{v(Q_r(z))}\int_{Q_r(z)}wdv\right)
	\left(\frac{1}{v(Q_r(z))}\int_{Q_r(z)}w^{-\frac{p'}{p}}dv\right)^{\frac{p}{p'}}<\infty.$$
The class $A^{\res}_p$ was introduced by Isralowitz \cite{Is} for characterizing the boundedness of Fock projections on the weighted spaces $L^p_{\alpha,w}$. Recall that for $\alpha>0$, the Fock projection $P_{\alpha}$ is defined for $f\in L^1_{\mathrm{loc}}(\C^n,dv)$ by
$$P_{\alpha}(f)(z):=\left(\frac{\alpha}{\pi}\right)^n\int_{\C^n}f(u)e^{\alpha\langle z,u\rangle}e^{-\alpha|u|^2}dv(u),\quad z\in\C^n.$$
Isralowitz \cite{Is} proved that for $p>1$, $P_{\alpha}$ is bounded on $L^p_{\alpha,w}$ if and only if $w\in A^{\res}_{p}$. Later on, Cascante, F\`{a}brega and Pel\'{a}ez \cite{CFP} generalized Isralowitz's result to the case $p=1$. Let the class $A^{\res}_1$ consist of weights $w$ on $\C^n$ such that $w>0$ a.e. on $\C^n$ and for some (or any) fixed $r>0$,
$$\sup_{z\in\C^n}\frac{\int_{Q_r(z)}wdv}{v(Q_r(z))\mathrm{ess}\inf_{u\in Q_r(z)}w(u)}<\infty.$$
It was proved in \cite{CFP} that $P_{\alpha}$ is bounded on $L^1_{\alpha,w}$ if and only if $w\in A^{\res}_1$. Similar to the Muckenhoupt weights, we define
$$A^{\res}_{\infty}:=\bigcup_{1\leq p<\infty}A^{\res}_p.$$
Recently, the function and operator theory on weighted Fock spaces induced by weights from $A^{\res}_{\infty}$ developed quickly; see \cite{CFP23,CFP,Ch24,CHW,CW24}.

In this paper, we are going to investigate a class of nonlinear Berezin-type operators acting on weighted Fock spaces $F^p_{\alpha,w}$ induced by $w\in A^{\res}_{\infty}$. Let $0<t,\alpha,\beta<\infty$ and $\mu$ be a positive Borel measure on $\C^n$. The Berezin-type operator $S^{t,\alpha,\beta}_{\mu}$ is formally defined for $f\in\h(\C^n)$ by
$$S^{t,\alpha,\beta}_{\mu}f(z):=\left(\int_{\C^n}e^{-\frac{\beta}{2}|z-u|^2}|f(u)|^te^{-\frac{\alpha t}{2}|u|^2}d\mu(u)\right)^{1/t}
,\quad z\in\C^n.$$
This class of operators was initially introduced by Pau and Zhao \cite{PZ15} in the setting of weighted Bergman spaces over the unit ball of $\C^n$, in route of establishing the product-type characterization of Bergman--Carleson measures. Zhou, Zhao and Tang \cite{ZZT} investigated the boundedness of the Berezin-type operators $S^{t,\alpha,\beta}_{\mu}$ on standard Fock spaces $F^p_{\alpha}$ and obtained the product-type characterization of Fock--Carleson measures. It was shown in \cite[Theorem 1.3]{ZZT} that if $S^{t,\alpha,\beta}_{\mu}:F^p_{\alpha}\to L^q(dv)$ is bounded, then
$$\sup_{z\in\C^n}\mu(B_1(z))<\infty$$
when $p\leq q$, and
$$\int_{\C^n}\mu(B_1(z))^{\frac{pq}{t(p-q)}}dv(z)<\infty$$
when $p>q$. Here and in the sequel, $B_r(z)$ denotes the Euclidean ball in $\C^n$ centered at $z\in\C^n$ with radius $r>0$. However, the sufficient condition for $S^{t,\alpha,\beta}_{\mu}:F^p_{\alpha}\to L^q(dv)$ to be bounded was left open in \cite{ZZT}. The purpose of this paper is to solve this problem.

Since the operator $S^{t,\alpha,\beta}_{\mu}$ is nonlinear, we here say a few words to clarify its boundedness and compactness. It is easy to verify that $S^{t,\alpha,\beta}_{\mu}:F^p_{\alpha,w}\to L^q(wdv)$ is bounded if and only if there exists $C>0$ such that for any $f\in F^p_{\alpha,w}$, $\left\|S^{t,\alpha,\beta}_{\mu}f\right\|_{L^q(wdv)}\leq C\|f\|_{F^p_{\alpha,w}}$. We will write
$$\left\|S^{t,\alpha,\beta}_{\mu}\right\|_{F^p_{\alpha,w}\to L^q(wdv)}:
=\sup_{\|f\|_{F^p_{\alpha,w}}\leq1}\left\|S^{t,\alpha,\beta}_{\mu}f\right\|_{L^q(wdv)}.$$
We say that $S^{t,\alpha,\beta}_{\mu}:F^p_{\alpha,w}\to L^q(wdv)$ is compact if $S^{t,\alpha,\beta}_{\mu}$ maps any bounded set in $F^p_{\alpha,w}$ to a relative compact set in $L^q(wdv)$.

For measurable subsets $E\subset\C^n$, we write $w(E):=\int_Ewdv$. Our main result reads as follows, which completely characterizes the boundedness and compactness of $S^{t,\alpha,\beta}_{\mu}:F^p_{\alpha,w}\to L^q(wdv)$ for all $0<p,q,t,\alpha,\beta<\infty$ and $w\in A^{\res}_{\infty}$. In particular, when $w\equiv1$, we solve the open problem raised in \cite{ZZT}.

\begin{theorem}\label{main1}
Let $0<p,q,t,\alpha,\beta<\infty$, $w\in A^{\res}_{\infty}$, and let $\mu$ be a positive Borel measure on $\C^n$. Consider the function
$$G_{\mu}(z):=\frac{\mu(B_1(z))^{\frac{1}{t}}}{w(B_1(z))^{\frac{1}{p}-\frac{1}{q}}},\quad z\in\C^n.$$
\begin{enumerate}
	\item [(1)] If $p\leq q$, then the following assertions hold.
	\begin{enumerate}
		\item [(i)] $S^{t,\alpha,\beta}_{\mu}:F^p_{\alpha,w}\to L^q(wdv)$ is bounded if and only if $G_{\mu}$ is bounded on $\C^n$. Moreover,
		$$\left\|S^{t,\alpha,\beta}_{\mu}\right\|_{F^p_{\alpha,w}\to L^q(wdv)}\asymp\sup_{z\in\C^n}G_{\mu}(z).$$
		\item [(ii)] Assume that $\mu$ is finite on compact subsets of $\C^n$. Then $S^{t,\alpha,\beta}_{\mu}:F^p_{\alpha,w}\to L^q(wdv)$ is compact if and only if $\lim_{|z|\to\infty}G_{\mu}(z)=0$.
	\end{enumerate}
	\item [(2)] If $p>q$, then the following conditions are equivalent:
	\begin{enumerate}
		\item [(i)] $S^{t,\alpha,\beta}_{\mu}:F^p_{\alpha,w}\to L^q(wdv)$ is bounded;
		\item [(ii)] $S^{t,\alpha,\beta}_{\mu}:F^p_{\alpha,w}\to L^q(wdv)$ is compact;
		\item [(iii)] $G_{\mu}\in L^{\frac{pq}{p-q}}(dv)$.
	\end{enumerate}
	Moreover,
	$$\left\|S^{t,\alpha,\beta}_{\mu}\right\|_{F^p_{\alpha,w}\to L^q(wdv)}\asymp\|G_{\mu}\|_{L^{\frac{pq}{p-q}}(dv)}.$$
\end{enumerate}
\end{theorem}

Two remarks are given in order.
\begin{enumerate}
	\item [(1)] When we consider the compactness of $S^{t,\alpha,\beta}_{\mu}:F^p_{\alpha,w}\to L^q(wdv)$ in the case $p\leq q$, the assumption that $\mu$ is finite on compact subsets of $\C^n$ is not a real restriction. In fact, by the characterization of the boundedness, we know that if $S^{t,\alpha,\beta}_{\mu}:F^p_{\alpha,w}\to L^q(wdv)$ is bounded, then $\mu$ must be finite on any compact subset of $\C^n$.
	\item [(2)] The Euclidean ball $B_1(z)$ in Theorem \ref{main1} can be replaced by $B_r(z)$ for any fixed $r>0$. 
\end{enumerate}

The main step to prove Theorem \ref{main1} is a discretization estimate for the quasi-norms of $S^{t,\alpha,\beta}_{\mu}f$ (see Proposition \ref{suff}), which is workable since the Gaussian weights $e^{-\alpha|z|^2}$ decay rapidly. The compactness needs more works due to the nonlinearity of $S^{t,\alpha,\beta}_{\mu}$. We overcome this obstacle and establish an equivalent description for the compactness of $S^{t,\alpha,\beta}_{\mu}:F^p_{\alpha,w}\to L^q(wdv)$ (see Proposition \ref{cpt-eq}), which helps us to finish the proof of Theorem \ref{main1}.

As an application of Theorem \ref{main1}, we can determine the boundedness and compactness of Toeplitz-type operators acting on weighted Fock spaces $F^p_{\alpha,w}$. Given $0<\alpha<\infty$ and a positive Borel measure $\mu$ on $\C^n$, the Toeplitz-type operator $T^{\alpha}_{\mu}$ is formally defined by
$$T^{\alpha}_{\mu}f(z):=\int_{\C^n}f(u)e^{\alpha\langle z,u\rangle-\alpha|u|^2}d\mu(u),\quad z\in\C^n,$$
where $f\in\h(\C^n)$. This operator has been studied extensively on standard Fock spaces; see \cite{BCI,Fu,HL,IZ,Me,WZ} and the references therein. We here consider the boundedness and compactness of $T^{\alpha}_{\mu}:F^p_{\alpha,w}\to F^q_{\alpha,w}$ with $w\in A^{\res}_{\infty}$. It is easy to see that if $S^{1,\alpha,\alpha}_{\mu}f$ is well-defined, then $T^{\alpha}_{\mu}f$ is well-defined, and
\begin{equation}\label{bt}
\int_{\C^n}\left|T^{\alpha}_{\mu}f(z)\right|^qe^{-\frac{q\alpha}{2}|z|^2}w(z)dv(z)\leq\int_{\C^n}S^{1,\alpha,\alpha}_{\mu}f(z)^qw(z)dv(z).
\end{equation}
Based on this observation, we can apply Theorem \ref{main1} to obtain the following result.

\begin{theorem}\label{main2}
Let $0<p,q,\alpha<\infty$, $w\in A^{\res}_{\infty}$, and let $\mu$ be a positive Borel measure on $\C^n$. Consider the function
$$H_{\mu}(z):=\frac{\mu(B_1(z))}{w(B_1(z))^{\frac{1}{p}-\frac{1}{q}}},\quad z\in\C^n.$$
\begin{enumerate}
	\item [(1)] If $p\leq q$, then the following assertions hold.
	\begin{enumerate}
		\item [(i)] $T^{\alpha}_{\mu}:F^p_{\alpha,w}\to F^q_{\alpha,w}$ is bounded if and only if $H_{\mu}$ is bounded. Moreover,
		$$\left\|T^{\alpha}_{\mu}\right\|_{F^p_{\alpha,w}\to F^q_{\alpha,w}}\asymp\sup_{z\in\C^n}H_{\mu}(z).$$
		\item [(ii)] Assume that $\mu$ is finite on compact subsets of $\C^n$. Then $T^{\alpha}_{\mu}:F^p_{\alpha,w}\to F^q_{\alpha,w}$ is compact if and only if $\lim_{|z|\to\infty}H_{\mu}(z)=0$.
	\end{enumerate}
	\item [(2)] If $p>q$, then the following conditions are equivalent:
	\begin{enumerate}
		\item [(i)] $T^{\alpha}_{\mu}:F^p_{\alpha,w}\to F^q_{\alpha,w}$ is bounded;
		\item [(ii)] $T^{\alpha}_{\mu}:F^p_{\alpha,w}\to F^q_{\alpha,w}$ is compact;
		\item [(iii)] $H_{\mu}\in L^{\frac{pq}{p-q}}(dv)$.
	\end{enumerate}
	Moreover,
	$$\left\|T^{\alpha}_{\mu}\right\|_{F^p_{\alpha,w}\to F^q_{\alpha,w}}\asymp\|H_{\mu}\|_{L^{\frac{pq}{p-q}}(dv)}.$$
\end{enumerate}
\end{theorem}

This paper is organized as follows. In Section \ref{pre}, we give some preliminary results that will be used in the sequel. Theorems \ref{main1} and \ref{main2} are proved in Sections \ref{proof1} and \ref{proof2}, respectively. Finally, in Section \ref{CR}, we give a characterization for $(w,\lambda)$-Fock--Carleson measures by products of functions in weighted Fock spaces $F^p_{\alpha,w}$, where $w\in A^{\res}_{\infty}$.

Throughout the paper, the notation $A\lesssim B$ (or $B\gtrsim A$) means that there exists a nonessential constant $C>0$ such that $A\leq CB$. If $A\lesssim B\lesssim A$, then we write $A\asymp B$. For a subset $E\subset\C^n$, $\chi_E$ denotes its characteristic function.

\section{Preliminaries}\label{pre}

In this section, we give some preliminary results that will be used in the sequel.

We first recall the following estimates, which can be found in \cite[Lemmas 3.2 and 3.4]{Is}. Here and in the sequel, for any $r>0$, we treat $r\Z^{2n}$ as a subset of $\C^{n}$ in the natural way.

\begin{lemma}\label{esti}
Let $w\in A^{\mathrm{restricted}}_{\infty}$ and $r>0$.
\begin{enumerate}
	\item [(1)] There exists $C_1>0$ such that for any $z\in\C^n$,
	$$w(Q_{3r}(z))\leq C_1w(Q_r(z)).$$
	\item [(2)] There exists $C_2>0$ such that for any $\nu,\nu'\in r\Z^{2n}$,
	$$\frac{w(Q_r(\nu))}{w(Q_r(\nu'))}\leq C_2^{|\nu-\nu'|}.$$
\end{enumerate}
\end{lemma}

\begin{remark}\label{remark}
There are two useful consequences of the above lemma.

(1) If $w\in A^{\res}_{\infty}$, then for any fixed $M,N\geq1$ and $r>0$,
$$w(Q_r(z))\asymp w(Q_{Nr}(u))\asymp w(B_r(z))\asymp w(B_{Nr}(u))$$
whenever $z,u\in\C^n$ satisfy $|z-u|<Mr$.

(2) For any $w\in A^{\res}_{\infty}$ and $\alpha>0$,
$$\int_{\C^n}e^{-\alpha|z|^2}w(z)dv(z)<\infty.$$
In fact, noting that for any $\nu\in\Z^{2n}$ and $z\in Q_1(\nu)$, $|\nu|^2\leq n+2|z|^2$, we can use Lemma \ref{esti} to obtain that
\begin{align*}
\int_{\C^n}e^{-\alpha|z|^2}w(z)dv(z)&=\sum_{\nu\in\Z^{2n}}\int_{Q_1(\nu)}e^{-\alpha|z|^2}w(z)dv(z)\\
&\lesssim\sum_{\nu\in\Z^{2n}}e^{-\frac{\alpha}{2}|\nu|^2}w(Q_1(\nu))\\
&\leq\sum_{\nu\in\Z^{2n}}e^{-\frac{\alpha}{2}|\nu|^2}C^{|\nu|}w(Q_1(0))\\
&\lesssim w(Q_1(0))<\infty.
\end{align*}
\end{remark}

The following lemma establishes some pointwise estimates for entire functions, which was proved in \cite[Lemma 3.1]{CFP} in the setting of one complex variable. The case of several complex variables is similar.

\begin{lemma}\label{pointwise}
Let $\alpha,p,r>0$ and $w\in A^{\res}_{\infty}$. Then for $f\in\h(\C^n)$ and $z\in\C^n$,
$$|f(z)|^pe^{-\frac{p\alpha}{2}|z|^2}\lesssim
\frac{1}{w(B_r(z))}\int_{B_r(z)}|f(u)|^pe^{-\frac{p\alpha}{2}|u|^2}w(u)dv(u),$$
where the implicit constant is independent of $f$ and $z$.
\end{lemma}

For $u\in\C^n$, we write $K^{\alpha}_u(z):=e^{\alpha\langle z,u\rangle}$ for the reproducing kernels of the standard Fock space $F^2_{\alpha}$. The following lemma gives some test functions in $F^p_{\alpha,w}$; see for instance \cite[Propositions 4.1 and 4.2]{CFP}. 

\begin{lemma}\label{test}
Let $\alpha,p>0$ and $w\in A^{\res}_{\infty}$.
\begin{enumerate}
  \item [(1)] For any $u\in\C^n$,
  $$\left\|K^{\alpha}_u\right\|^p_{F^p_{\alpha,w}}\asymp e^{\frac{p\alpha}{2}|u|^2}w(B_1(u)).$$
  \item [(2)] For any sequence $c=\{c_{\nu}\}_{\nu\in\Z^{2n}}\in l^p(\Z^{2n})$, the function
  $$f:=\sum_{\nu\in\Z^{2n}}c_{\nu}\frac{K^{\alpha}_{\nu}}{\left\|K^{\alpha}_{\nu}\right\|_{F^p_{\alpha,w}}}$$
  belongs to $F^p_{\alpha,w}$ with $\|f\|_{F^p_{\alpha,w}}\lesssim\|c\|_{l^p(\Z^{2n})}$.
\end{enumerate}
\end{lemma}

We end this section by the following discretization lemmas.

\begin{lemma}\label{dis1}
Let $\gamma>0$, $\eta\in\mathbb{R}$, $w\in A^{\res}_{\infty}$, and let $\mu$ be a positive Borel measure on $\C^n$. Then
$$\sup_{z\in\C^n}\frac{\mu(B_1(z))^{\gamma}}{w(B_1(z))^{\eta}}\asymp\sup_{\nu\in\Z^{2n}}\frac{\mu(Q_1(\nu))^{\gamma}}{w(Q_1(\nu))^{\eta}}.$$
\end{lemma}
\begin{proof}
We first consider the inequality ``$\gtrsim$''. Fix $\nu\in\Z^{2n}$ and write
$$I_{\nu}:=\left\{\nu'\in\sqrt{\frac{1}{n}}\Z^{2n}:Q_{\sqrt{\frac{1}{n}}}(\nu')\cap Q_1(\nu)\neq\emptyset\right\}.$$
Then we have
$$Q_1(\nu)\subset\bigcup_{\nu'\in I_{\nu}}Q_{\sqrt{\frac{1}{n}}}(\nu')\subset\bigcup_{\nu'\in I_{\nu}}B_1(\nu').$$
It is easy to verify that for $\nu'\in I_{\nu}$, $Q_{\sqrt{\frac{1}{n}}}(\nu')\subset Q_3(\nu)$, which implies that
$$|I_{\nu}|\leq3^{2n}n^n.$$
Consequently,
$$\mu(Q_1(\nu))^{\gamma}\leq\left(\sum_{\nu'\in I_{\nu}}\mu(B_1(\nu'))\right)^{\gamma}\lesssim\sum_{\nu'\in I_{\nu}}\mu(B_1(\nu'))^{\gamma},$$
which together with Remark \ref{remark} gives that
\begin{equation}\label{in1}
\frac{\mu(Q_1(\nu))^{\gamma}}{w(Q_1(\nu))^{\eta}}\lesssim\sum_{\nu'\in I_{\nu}}\frac{\mu(B_1(\nu'))^{\gamma}}{w(Q_1(\nu))^{\eta}}
\asymp\sum_{\nu'\in I_{\nu}}\frac{\mu(B_1(\nu'))^{\gamma}}{w(B_1(\nu'))^{\eta}}.
\end{equation}
Since $\nu\in\Z^{2n}$ is arbitrary, we obtain that
$$\sup_{\nu\in\Z^{2n}}\frac{\mu(Q_1(\nu))^{\gamma}}{w(Q_1(\nu))^{\eta}}
\lesssim\sup_{\nu\in\Z^{2n}}\sum_{\nu'\in I_{\nu}}\frac{\mu(B_1(\nu'))^{\gamma}}{w(B_1(\nu'))^{\eta}}
\leq3^{2n}n^n\sup_{z\in\C^n}\frac{\mu(B_1(z))^{\gamma}}{w(B_1(z))^{\eta}}.$$

We next consider the inequality ``$\lesssim$''. Fix $z\in\C^n$ and write
$$\Gamma_z:=\left\{\nu\in\Z^{2n}:B_1(z)\cap Q_1(\nu)\neq\emptyset\right\}.$$
Then $B_1(z)\subset\bigcup_{\nu\in\Gamma_z}Q_1(\nu)$, and for any $\nu\in\Gamma_z$, $B_{\frac{1}{2}}(\nu)\subset B_{3\sqrt{n}}(z)$, which implies that
$$|\Gamma_z|\leq6^{2n}n^n.$$
Consequently,
$$\mu(B_1(z))^{\gamma}\leq\left(\sum_{\nu\in\Gamma_z}\mu(Q_1(\nu))\right)^{\gamma}\lesssim\sum_{\nu\in\Gamma_z}\mu(Q_1(\nu))^{\gamma},$$
which, in conjunction with Remark \ref{remark}, yields that
\begin{equation}\label{in2}
\frac{\mu(B_1(z))^{\gamma}}{w(B_1(z))^{\eta}}\lesssim\sum_{\nu\in\Gamma_z}\frac{\mu(Q_1(\nu))^{\gamma}}{w(B_1(z))^{\eta}}
\asymp\sum_{\nu\in\Gamma_z}\frac{\mu(Q_1(\nu))^{\gamma}}{w(Q_1(\nu))^{\eta}}.
\end{equation}
Therefore,
$$\sup_{z\in\C^n}\frac{\mu(B_1(z))^{\gamma}}{w(B_1(z))^{\eta}}
\lesssim\sup_{z\in\C^n}\sum_{\nu\in\Gamma_z}\frac{\mu(Q_1(\nu))^{\gamma}}{w(Q_1(\nu))^{\eta}}
\leq 6^{2n}n^n\sup_{\nu\in\Z^{2n}}\frac{\mu(Q_1(\nu))^{\gamma}}{w(Q_1(\nu))^{\eta}}.$$
The proof is complete.
\end{proof}

\begin{lemma}\label{dis2}
Let $\gamma>0$, $\eta\in\mathbb{R}$, $w\in A^{\res}_{\infty}$, and let $\mu$ be a positive Borel measure on $\C^n$. Then
$$\int_{\C^n}\frac{\mu(B_1(z))^{\gamma}}{w(B_1(z))^{\eta}}dv(z)\asymp
\sum_{\nu\in\Z^{2n}}\frac{\mu(Q_1(\nu))^{\gamma}}{w(Q_1(\nu))^{\eta}}.$$
\begin{proof}
Similar to \eqref{in1}, for any $\nu\in\Z^{2n}$, we have
$$\frac{\mu(Q_1(\nu))^{\gamma}}{w(Q_1(\nu))^{\eta}}
\lesssim\sum_{\nu'\in\tilde{I}_{\nu}}\frac{\mu(B_{1/2}(\nu'))^{\gamma}}{w(B_{1/2}(\nu'))^{\eta}},$$
where
$$\tilde{I}_{\nu}=\left\{\nu'\in\sqrt{\frac{1}{4n}}\Z^{2n}:Q_{\sqrt{\frac{1}{4n}}}(\nu')\cap Q_1(\nu)\neq\emptyset\right\}.$$
Arguing as in the proof of Lemma \ref{dis1}, we know that for any $\nu'\in\sqrt{\frac{1}{4n}}\Z^{2n}$, the cardinal of the set $\{\nu\in\Z^{2n}:\nu'\in\tilde{I}_{\nu}\}$ is bounded by some constant depending only on $n$. Therefore, noting that for any $z\in B_{1/2}(\nu')$, $B_{1/2}(\nu')\subset B_1(z)$, we establish that
\begin{align*}
\sum_{\nu\in\Z^{2n}}\frac{\mu(Q_1(\nu))^{\gamma}}{w(Q_1(\nu))^{\eta}}
&\lesssim\sum_{\nu\in\Z^{2n}}\sum_{\nu'\in\tilde{I}_{\nu}}\frac{\mu(B_{1/2}(\nu'))^{\gamma}}{w(B_{1/2}(\nu'))^{\eta}}\\
&\asymp\sum_{\nu'\in\sqrt{\frac{1}{4n}}\Z^{2n}}\sum_{\nu\in\Z^{2n}:\nu'\in\tilde{I}_{\nu}}
    \int_{B_{1/2}(\nu')}\frac{\mu(B_{1/2}(\nu'))^{\gamma}}{w(B_{1/2}(\nu'))^{\eta}}dv(z)\\
&\lesssim\sum_{\nu'\in\sqrt{\frac{1}{4n}}\Z^{2n}}\int_{B_{1/2}(\nu')}\frac{\mu(B_{1}(z))^{\gamma}}{w(B_{1}(z))^{\eta}}dv(z)\\
&\lesssim\int_{\C^n}\frac{\mu(B_{1}(z))^{\gamma}}{w(B_{1}(z))^{\eta}}dv(z).
\end{align*}

Conversely, by \eqref{in2},
\begin{align*}
\int_{\C^n}\frac{\mu(B_1(z))^{\gamma}}{w(B_1(z))^{\eta}}dv(z)
&\lesssim\int_{\C^n}\sum_{\nu\in\Gamma_z}\frac{\mu(Q_1(\nu))^{\gamma}}{w(Q_1(\nu))^{\eta}}dv(z)\\
&=\sum_{\nu\in\Z^{2n}}\int_{\{z\in\C^n:B_1(z)\cap Q_1(\nu)\neq\emptyset\}}
  \frac{\mu(Q_1(\nu))^{\gamma}}{w(Q_1(\nu))^{\eta}}dv(z)\\
&\lesssim\sum_{\nu\in\Z^{2n}}\frac{\mu(Q_1(\nu))^{\gamma}}{w(Q_1(\nu))^{\eta}},
\end{align*}
which finishes the proof.
\end{proof}
\end{lemma}

\section{Proof of Theorem \ref{main1}}\label{proof1}

The purpose of this section is to prove Theorem \ref{main1}, which is based on the following proposition.

\begin{proposition}\label{suff}
Let $0<p,q,t,\alpha,\beta<\infty$, $w\in A^{\res}_{\infty}$, and let $\mu$ be a positive Borel measure on $\C^n$. Then for any $f\in F^p_{\alpha,w}$,
\begin{align*}
\int_{\C^n}&S^{t,\alpha,\beta}_{\mu}f(z)^qw(z)dv(z)\\
&\lesssim\sum_{\nu\in\Z^{2n}}\frac{\mu(Q_1(\nu))^{\frac{q}{t}}}{w(Q_1(\nu))^{\frac{q}{p}-1}}
\left(\int_{B_{2n}(\nu)}|f(u)|^pe^{-\frac{p\alpha}{2}|u|^2}w(u)dv(u)\right)^{\frac{q}{p}},
\end{align*}
where the implicit constant is independent of $f$ and $\mu$.
\end{proposition}
\begin{proof}
Note that for any $\nu,\nu'\in\Z^{2n}$ and $z\in Q_1(\nu')$, $\xi\in Q_1(\nu)$, $e^{-\frac{\beta}{2}|z-\xi|^2}\lesssim e^{-\frac{\beta}{4}|\nu-\nu'|^2}$. We have that
\begin{align*}
&\int_{\C^n}S^{t,\alpha,\beta}_{\mu}f(z)^qw(z)dv(z)\\
&\ =\int_{\C^n}\left(\int_{\C^n}e^{-\frac{\beta}{2}|z-\xi|^2}|f(\xi)|^te^{-\frac{\alpha t}{2}|\xi|^2}d\mu(\xi)\right)^{\frac{q}{t}}w(z)dv(z)\\
&\ \lesssim\sum_{\nu'\in\Z^{2n}}\int_{Q_1(\nu')}\left(\sum_{\nu\in\Z^{2n}}e^{-\frac{\beta}{4}|\nu-\nu'|^2}
  \int_{Q_1(\nu)}|f(\xi)|^te^{-\frac{\alpha t}{2}|\xi|^2}d\mu(\xi)\right)^{\frac{q}{t}}w(z)dv(z)\\
&\ =\sum_{\nu'\in\Z^{2n}}w(Q_1(\nu'))\left(\sum_{\nu\in\Z^{2n}}e^{-\frac{\beta}{4}|\nu-\nu'|^2}
  \int_{Q_1(\nu)}|f(\xi)|^te^{-\frac{\alpha t}{2}|\xi|^2}d\mu(\xi)\right)^{\frac{q}{t}}.
\end{align*}
Lemma \ref{pointwise} together with Remark \ref{remark} implies that for any $\xi\in Q_1(\nu)$,
\begin{align*}
|f(\xi)|^te^{-\frac{\alpha t}{2}|\xi|^2}
&\lesssim\left(\frac{1}{w(B_1(\xi))}\int_{B_1(\xi)}|f(u)|^pe^{-\frac{p\alpha}{2}|u|^2}w(u)dv(u)\right)^{\frac{t}{p}}\\
&\lesssim\frac{1}{w(Q_1(\nu))^{t/p}}\left(\int_{B_{2n}(\nu)}|f(u)|^pe^{-\frac{p\alpha}{2}|u|^2}w(u)dv(u)\right)^{\frac{t}{p}}.
\end{align*}
Therefore,
\begin{align*}
\int_{\C^n}&S^{t,\alpha,\beta}_{\mu}f(z)^qw(z)dv(z)\\
&\lesssim\sum_{\nu'\in\Z^{2n}}w(Q_1(\nu'))\left(\sum_{\nu\in\Z^{2n}}e^{-\frac{\beta}{4}|\nu-\nu'|^2}
  \frac{\mu(Q_1(\nu))}{w(Q_1(\nu))^{\frac{t}{p}}}
  \widehat{|f|^p}(\nu)^{\frac{t}{p}}\right)^{\frac{q}{t}},
\end{align*}
where
$$\widehat{|f|^p}(\nu)=\int_{B_{2n}(\nu)}|f(u)|^pe^{-\frac{p\alpha}{2}|u|^2}w(u)dv(u).$$
In the case $q>t$, we can apply H\"{o}lder's inequality and Lemma \ref{esti} to obtain that
\begin{align*}
&\int_{\C^n}S^{t,\alpha,\beta}_{\mu}f(z)^qw(z)dv(z)\\
&\ \lesssim\sum_{\nu'\in\Z^{2n}}w(Q_1(\nu'))\left(\sum_{\nu\in\Z^{2n}}e^{-\frac{\beta}{4}|\nu-\nu'|^2}
    \frac{\mu(Q_1(\nu))}{w(Q_1(\nu))^{\frac{t}{p}}}
    \widehat{|f|^p}(\nu)^{\frac{t}{p}}\right)^{\frac{q}{t}}\\
&\ \leq\sum_{\nu'\in\Z^{2n}}w(Q_1(\nu'))\left(\sum_{\nu\in\Z^{2n}}e^{-\frac{\beta}{4}|\nu-\nu'|^2}\right)^{\frac{q-t}{t}}
    \cdot\left(\sum_{\nu\in\Z^{2n}}e^{-\frac{\beta}{4}|\nu-\nu'|^2}\frac{\mu(Q_1(\nu))^{\frac{q}{t}}}{w(Q_1(\nu))^{\frac{q}{p}}}
    \widehat{|f|^p}(\nu)^{\frac{q}{p}}\right)\\
&\ \lesssim\sum_{\nu'\in\Z^{2n}}w(Q_1(\nu'))\sum_{\nu\in\Z^{2n}}e^{-\frac{\beta}{4}|\nu-\nu'|^2}
    \frac{\mu(Q_1(\nu))^{\frac{q}{t}}}{w(Q_1(\nu))^{\frac{q}{p}}}
    \widehat{|f|^p}(\nu)^{\frac{q}{p}}\\
&\ \leq\sum_{\nu\in\Z^{2n}}\sum_{\nu'\in\Z^{2n}}e^{-\frac{\beta}{4}|\nu-\nu'|^2}C^{|\nu-\nu'|}
    \frac{\mu(Q_1(\nu))^{\frac{q}{t}}}{w(Q_1(\nu))^{\frac{q}{p}-1}}
    \widehat{|f|^p}(\nu)^{\frac{q}{p}}\\
&\ \lesssim\sum_{\nu\in\Z^{2n}}\frac{\mu(Q_1(\nu))^{\frac{q}{t}}}{w(Q_1(\nu))^{\frac{q}{p}-1}}\widehat{|f|^p}(\nu)^{\frac{q}{p}}.
\end{align*}
In the case $q\leq t$, we have
\begin{align*}
&\int_{\C^n}S^{t,\alpha,\beta}_{\mu}f(z)^qw(z)dv(z)\\
&\ \lesssim\sum_{\nu'\in\Z^{2n}}w(Q_1(\nu'))\left(\sum_{\nu\in\Z^{2n}}e^{-\frac{\beta}{4}|\nu-\nu'|^2}
    \frac{\mu(Q_1(\nu))}{w(Q_1(\nu))^{\frac{t}{p}}}\widehat{|f|^p}(\nu)^{t/p}\right)^{\frac{q}{t}}\\
&\ \leq\sum_{\nu'\in\Z^{2n}}w(Q_1(\nu'))\sum_{\nu\in\Z^{2n}}e^{-\frac{q\beta}{4t}|\nu-\nu'|^2}
    \frac{\mu(Q_1(\nu))^{\frac{q}{t}}}{w(Q_1(\nu))^{\frac{q}{p}}}
    \widehat{|f|^p}(\nu)^{\frac{q}{p}}\\
&\ \leq\sum_{\nu\in\Z^{2n}}\sum_{\nu'\in\Z^{2n}}e^{-\frac{q\beta}{4t}|\nu-\nu'|^2}C^{|\nu-\nu'|}
    \frac{\mu(Q_1(\nu))^{\frac{q}{t}}}{w(Q_1(\nu))^{\frac{q}{p}-1}}\widehat{|f|^p}(\nu)^{\frac{q}{p}}\\
&\ \lesssim\sum_{\nu\in\Z^{2n}}\frac{\mu(Q_1(\nu))^{\frac{q}{t}}}{w(Q_1(\nu))^{\frac{q}{p}-1}}\widehat{|f|^p}(\nu)^{\frac{q}{p}}.
\end{align*}
The proof is complete.
\end{proof}

\subsection{Boundedness}

Recall that  the function $G_{\mu}$ is defined by
$$G_{\mu}(z)=\frac{\mu(B_1(z))^{\frac{1}{t}}}{w(B_1(z))^{\frac{1}{p}-\frac{1}{q}}},\quad z\in\C^n.$$
The following theorem characterize the boundedness of $S^{t,\alpha,\beta}_{\mu}:F^p_{\alpha,w}\to L^q(wdv)$ in the case $p\leq q$.

\begin{theorem}\label{bdd<}
Let $0<p\leq q<\infty$, $0<t,\alpha,\beta<\infty$, $w\in A^{\res}_{\infty}$, and let $\mu$ be a positive Borel measure on $\C^n$. Then $S^{t,\alpha,\beta}_{\mu}:F^p_{\alpha,w}\to L^q(wdv)$ is bounded if and only if $G_{\mu}$ is bounded on $\C^n$. Moreover,
$$\left\|S^{t,\alpha,\beta}_{\mu}\right\|_{F^p_{\alpha,w}\to L^q(wdv)}\asymp\sup_{z\in\C^n}G_{\mu}(z).$$
\end{theorem}
\begin{proof}
Suppose first that $S^{t,\alpha,\beta}_{\mu}:F^p_{\alpha,w}\to L^q(wdv)$ is bounded. For any $z\in\C^n$, define
$$f_z(u):=\frac{e^{\alpha\langle u,z\rangle-\frac{\alpha}{2}|z|^2}}{w(B_1(z))^{1/p}},\quad u\in\C^n.$$
Then by Lemma \ref{test}, $f_z\in F^p_{\alpha,w}$ with $\|f_z\|_{F^p_{\alpha,w}}\asymp1$. Consequently,
\begin{align}\label{bdd-nece}
\nonumber\left\|S^{t,\alpha,\beta}_{\mu}\right\|^q_{F^p_{\alpha,w}\to L^q(wdv)}
\nonumber&\gtrsim\int_{\C^n}S^{t,\alpha,\beta}_{\mu}f_z(\xi)^qw(\xi)dv(\xi)\\
\nonumber&=\int_{\C^n}\left(\int_{\C^n}e^{-\frac{\beta}{2}|\xi-u|^2}|f_z(u)|^te^{-\frac{\alpha t}{2}|u|^2}d\mu(u)\right)^{q/t}w(\xi)dv(\xi)\\
\nonumber&\geq\int_{B_1(z)}\left(\int_{B_1(z)}\frac{e^{-\frac{\beta}{2}|\xi-u|^2-\frac{\alpha t}{2}|z-u|^2}}{w(B_1(z))^{t/p}}
    d\mu(u)\right)^{q/t}w(\xi)dv(\xi)\\
&\asymp\frac{\mu(B_1(z))^{\frac{q}{t}}}{w(B_1(z))^{\frac{q}{p}-1}}.
\end{align}
Therefore, $G_{\mu}$ is bounded on $\C^n$, and $\sup_{z\in\C^n}G_{\mu}(z)\lesssim\left\|S^{t,\alpha,\beta}_{\mu}\right\|_{F^p_{\alpha,w}\to L^q(wdv)}$.

Suppose next that $G_{\mu}$ is bounded on $\C^n$. Note that $q/p\geq1$. For any $f\in F^p_{\alpha,w}$, we can apply Proposition \ref{suff} and Lemma \ref{dis1} to obtain that
\begin{align*}
&\int_{\C^n}S^{t,\alpha,\beta}_{\mu}f(z)^qw(z)dv(z)\\
&\ \lesssim\sum_{\nu\in\Z^{2n}}\frac{\mu(Q_1(\nu))^{\frac{q}{t}}}{w(Q_1(\nu))^{\frac{q}{p}-1}}
    \left(\int_{B_{2n}(\nu)}|f(u)|^pe^{-\frac{p\alpha}{2}|u|^2}w(u)dv(u)\right)^{q/p}\\
&\ \leq\sup_{\nu\in\Z^{2n}}\frac{\mu(Q_1(\nu))^{\frac{q}{t}}}{w(Q_1(\nu))^{\frac{q}{p}-1}}\cdot
    \sum_{\nu\in\Z^{2n}}\left(\int_{B_{2n}(\nu)}|f(u)|^pe^{-\frac{p\alpha}{2}|u|^2}w(u)dv(u)\right)^{q/p}\\
&\ \lesssim\sup_{z\in\C^n}G_{\mu}(z)^q\cdot\left(\sum_{\nu\in\Z^{2n}}\int_{B_{2n}(\nu)}|f(u)|^pe^{-\frac{p\alpha}{2}|u|^2}w(u)dv(u)\right)^{q/p}\\
&\ \lesssim\sup_{z\in\C^n}G_{\mu}(z)^q\cdot\|f\|^q_{F^p_{\alpha,w}},
\end{align*}
which implies that $S^{t,\alpha,\beta}_{\mu}:F^p_{\alpha,w}\to L^q(wdv)$ is bounded, and $\left\|S^{t,\alpha,\beta}_{\mu}\right\|_{F^p_{\alpha,w}\to L^q(wdv)}\lesssim\sup_{z\in\C^n}G_{\mu}(z)$.
\end{proof}

The following theorem concerns the boundedness of $S^{t,\alpha,\beta}_{\mu}:F^p_{\alpha,w}\to L^q(wdv)$ in the case $p>q$.

\begin{theorem}\label{bdd>}
Let $0<q<p<\infty$, $0<t,\alpha,\beta<\infty$, $w\in A^{\res}_{\infty}$, and let $\mu$ be a positive Borel measure on $\C^n$. Then $S^{t,\alpha,\beta}_{\mu}:F^p_{\alpha,w}\to L^q(wdv)$ is bounded if and only if $G_{\mu}\in L^{\frac{pq}{p-q}}(dv)$. Moreover,
$$\left\|S^{t,\alpha,\beta}_{\mu}\right\|_{F^p_{\alpha,w}\to L^q(wdv)}\asymp\|G_{\mu}\|_{L^{\frac{pq}{p-q}}(dv)}.$$
\end{theorem}
\begin{proof}
Suppose first that $S^{t,\alpha,\beta}_{\mu}:F^p_{\alpha,w}\to L^q(wdv)$ is bounded. For any $c=\{c_{\nu}\}_{\nu\in\Z^{2n}}\in l^p(\Z^{2n})$, define
$$F_{\tau}(z):=\sum_{\nu\in\Z^{2n}}c_{\nu}r_{\nu}(\tau)f_{\nu},$$
where $\{r_{\nu}\}_{\nu\in \Z^{2n}}$ is a sequence of Rademacher functions on $[0,1]$ (see \cite[Appendix A]{Du}), and
$$f_{\nu}=\frac{K^{\alpha}_{\nu}}{\|K^{\alpha}_{\nu}\|_{F^p_{\alpha,w}}}.$$
Then by Lemma \ref{test}, for almost every $\tau\in[0,1]$, $F_{\tau}\in F^p_{\alpha,w}$, and $\|F_{\tau}\|_{F^p_{\alpha,w}}\lesssim\|c\|_{l^p(\Z^{2n})}$. Consequently, the boundedness of $S^{t,\alpha,\beta}_{\mu}:F^p_{\alpha,w}\to L^q(wdv)$ yields that
\begin{equation}\label{int}
\int_0^1\left\|S^{t,\alpha,\beta}_{\mu}F_{\tau}\right\|^q_{L^q(wdv)}d\tau
\lesssim\|S^{t,\alpha,\beta}_{\mu}\|_{F^p_{\alpha,w}\to L^q(wdv)}^q\cdot\|c\|^q_{l^p(\Z^{2n})}.
\end{equation}
Using Fubini's theorem and Kahane--Khinchine's inequalities (see \cite[Theorem 2.1]{Ka} and \cite[Appendix A]{Du}), we establish that
\begin{align*}
&\int_0^1\left\|S^{t,\alpha,\beta}_{\mu}F_{\tau}\right\|^q_{L^q(wdv)}d\tau\\
&\ =\int_{\C^n}\int_0^1\left(\int_{\C^n}e^{-\frac{\beta}{2}|z-\xi|^2}
    \left|\sum_{\nu\in\Z^{2n}}c_{\nu}r_{\nu}(\tau)f_{\nu}(\xi)\right|^t
    e^{-\frac{\alpha t}{2}|\xi|^2}d\mu(\xi)\right)^{q/t}d\tau w(z)dv(z)\\
&\ \asymp\int_{\C^n}\left(\int_0^1\int_{\C^n}e^{-\frac{\beta}{2}|z-\xi|^2}
    \left|\sum_{\nu\in\Z^{2n}}c_{\nu}r_{\nu}(\tau)f_{\nu}(\xi)\right|^t
    e^{-\frac{\alpha t}{2}|\xi|^2}d\mu(\xi)d\tau\right)^{q/t} w(z)dv(z)\\
&\ \asymp\int_{\C^n}\left(\int_{\C^n}e^{-\frac{\beta}{2}|z-\xi|^2}
    \left(\sum_{\nu\in\Z^{2n}}|c_{\nu}|^2|f_{\nu}(\xi)|^2\right)^{t/2}
    e^{-\frac{\alpha t}{2}|\xi|^2}d\mu(\xi)\right)^{q/t} w(z)dv(z)\\
&\ \geq\int_{\C^n}\left(\int_{\C^n}e^{-\frac{\beta}{2}|z-\xi|^2}
    \left(\sum_{\nu\in\Z^{2n}}|c_{\nu}|^2|f_{\nu}(\xi)|^2\chi_{Q_1(\nu)}(\xi)\right)^{t/2}
    e^{-\frac{\alpha t}{2}|\xi|^2}d\mu(\xi)\right)^{q/t} w(z)dv(z)\\
&\ =\int_{\C^n}\left(\sum_{\nu\in\Z^{2n}}|c_{\nu}|^t\int_{Q_1(\nu)}e^{-\frac{\beta}{2}|z-\xi|^2}
    |f_{\nu}(\xi)|^te^{-\frac{\alpha t}{2}|\xi|^2}d\mu(\xi)\right)^{q/t} w(z)dv(z).
\end{align*}
Note that by Lemma \ref{test} and Remark \ref{remark}, for $\xi\in Q_1(\nu)$,
$$|f_{\nu}(\xi)|^te^{-\frac{\alpha t}{2}|\xi|^2}\asymp\frac{e^{-\frac{\alpha t}{2}|\nu-\xi|^2}}{w(B_1(\nu))^{t/p}}\asymp
\frac{1}{w(Q_1(\nu))^{t/p}}.$$
Therefore, we obtain that
\begin{align*}
&\int_0^1\left\|S^{t,\alpha,\beta}_{\mu}F_{\tau}\right\|^q_{L^q(wdv)}d\tau\\
&\ \gtrsim\int_{\C^n}\left(\sum_{\nu\in\Z^{2n}}\frac{|c_{\nu}|^t}{w(Q_1(\nu))^{t/p}}\int_{Q_1(\nu)}e^{-\frac{\beta}{2}|z-\xi|^2}
    d\mu(\xi)\right)^{q/t} w(z)dv(z)\\
&\ \geq\int_{\C^n}\left(\sum_{\nu\in\Z^{2n}}\frac{|c_{\nu}|^t\chi_{Q_1(\nu)}(z)}{w(Q_1(\nu))^{t/p}}\int_{Q_1(\nu)}e^{-\frac{\beta}{2}|z-\xi|^2}
    d\mu(\xi)\right)^{q/t} w(z)dv(z)\\
&\ =\sum_{\nu\in\Z^{2n}}\frac{|c_{\nu}|^q}{w(Q_1(\nu))^{q/p}}
    \int_{Q_1(\nu)}\left(\int_{Q_1(\nu)}e^{-\frac{\beta}{2}|z-\xi|^2}d\mu(\xi)\right)^{q/t}w(z)dv(z)\\
&\ \asymp\sum_{\nu\in\Z^{2n}}|c_{\nu}|^q\cdot\frac{\mu(Q_1(\nu))^{\frac{q}{t}}}{w(Q_1(\nu))^{\frac{q}{p}-1}},
\end{align*}
which, in conjunction with \eqref{int}, implies that
$$\sum_{\nu\in\Z^{2n}}|c_{\nu}|^q\cdot\frac{\mu(Q_1(\nu))^{\frac{q}{t}}}{w(Q_1(\nu))^{\frac{q}{p}-1}}\lesssim
\|S^{t,\alpha,\beta}_{\mu}\|_{F^p_{\alpha,w}\to L^q(wdv)}^q\cdot\|c\|^q_{l^p(\Z^{2n})}.$$
Since $c=\{c_{\nu}\}_{\nu\in\Z^{2n}}\in l^p(\Z^{2n})$ is arbitrary, we may apply the duality to obtain that
$$\left(\sum_{\nu\in\Z^{2n}}\left(\frac{\mu(Q_1(\nu))^{\frac{q}{t}}}{w(Q_1(\nu))^{\frac{q}{p}-1}}\right)^{\frac{p}{p-q}}\right)^{\frac{p-q}{p}}\lesssim\|S^{t,\alpha,\beta}_{\mu}\|_{F^p_{\alpha,w}\to L^q(wdv)}^q.$$
Combining the above estimate with Lemma \ref{dis2}, we conclude that $G_{\mu}\in L^{\frac{pq}{p-q}}(dv)$, and $\|G_{\mu}\|_{L^{\frac{pq}{p-q}}(dv)}\lesssim\|S^{t,\alpha,\beta}_{\mu}\|_{F^p_{\alpha,w}\to L^q(wdv)}$.

Conversely, suppose that $G_{\mu}\in L^{\frac{pq}{p-q}}(dv)$. Then for any $f\in F^p_{\alpha,w}$, Proposition \ref{suff} together with H\"{o}lder's inequality and Lemma \ref{dis2} yields that
\begin{align*}
&\int_{\C^n}S^{t,\alpha,\beta}_{\mu}f(z)^qw(z)dv(z)\\
&\ \lesssim\sum_{\nu\in\Z^{2n}}\frac{\mu(Q_1(\nu))^{\frac{q}{t}}}{w(Q_1(\nu))^{\frac{q}{p}-1}}
\left(\int_{B_{2n}(\nu)}|f(u)|^pe^{-\frac{p\alpha}{2}|u|^2}w(u)dv(u)\right)^{\frac{q}{p}}\\
&\ \leq\left(\sum_{\nu\in\Z^{2n}}\left(\frac{\mu(Q_1(\nu))^{\frac{1}{t}}}
    {w(Q_1(\nu))^{\frac{1}{p}-\frac{1}{q}}}\right)^{\frac{pq}{p-q}}\right)^{\frac{p-q}{p}}\cdot
    \left(\sum_{\nu\in\Z^{2n}}\int_{B_{2n}(\nu)}|f(u)|^pe^{-\frac{p\alpha}{2}|u|^2}w(u)dv(u)\right)^{\frac{q}{p}}\\
&\ \lesssim\|G_{\mu}\|^q_{L^{\frac{pq}{p-q}}(dv)}\cdot\|f\|^q_{F^p_{\alpha,w}},
\end{align*}
which finishes the proof.
\end{proof}

\subsection{Compactness}

To characterize the compactness of $S^{t,\alpha,\beta}_{\mu}:F^p_{\alpha,w}\to L^q(wdv)$, we need an equivalent description, which is based on the following lemmas.

\begin{lemma}\label{uni0}
Let $0<p,q,t,\alpha,\beta<\infty$, $w\in A^{\res}_{\infty}$, and let $\mu$ be a positive Borel measure on $\C^n$. Suppose that  $S^{t,\alpha,\beta}_{\mu}:F^p_{\alpha,w}\to L^q(wdv)$ is bounded. Then for any bounded sequence $\{f_k\}\subset F^p_{\alpha,w}$ that converges to $0$ uniformly on compact subsets of $\C^n$, we have $S^{t,\alpha,\beta}_{\mu}f_k\to0$ uniformly on compact subsets of $\C^n$.
\end{lemma}
\begin{proof}
Fix $R_0>0$. We will finish the proof by showing that $S^{t,\alpha,\beta}_{\mu}f_k\to0$ uniformly on $B_{R_0}(0)$. Note that for $z\in B_{R_0}(0)$ and $\xi\in Q_1(\nu)$,
$$|z-\xi|^2\geq\frac{1}{2}|\nu|^2-(R_0+1)^2.$$
We can deduce from Lemma \ref{pointwise} and Remark \ref{remark} that for every $k\geq1$ and $z\in B_{R_0}(0)$,
\begin{align}\label{upp}
\nonumber&S^{t,\alpha,\beta}_{\mu}f_k(z)^t\\
\nonumber&\ =\int_{\C^n}e^{-\frac{\beta}{2}|z-\xi|^2}|f_k(\xi)|^te^{-\frac{\alpha t}{2}|\xi|^2}d\mu(\xi)\\
\nonumber&\ \lesssim\sum_{\nu\in\Z^{2n}}\int_{Q_1(\nu)}e^{-\frac{\beta}{2}|z-\xi|^2}
    \left(\frac{1}{w(B_1(\xi))}\int_{B_1(\xi)}|f_k(u)|^pe^{-\frac{p\alpha}{2}|u|^2}w(u)dv(u)\right)^{t/p}d\mu(\xi)\\
&\ \lesssim\sum_{\nu\in\Z^{2n}}e^{-\frac{\beta}{4}|\nu|^2}\frac{\mu(Q_1(\nu))}{w(Q_1(\nu))^{t/p}}
    \left(\int_{B_{2n}(\nu)}|f_k(u)|^pe^{-\frac{p\alpha}{2}|u|^2}w(u)dv(u)\right)^{t/p}.
\end{align}
We now separate into two cases: $p\leq q$ and $p>q$.

{\bf Case 1: $p\leq q$.} By Theorem \ref{bdd<} and Lemma \ref{dis1}, we have that for any $\nu\in\Z^{2n}$,
$$\frac{\mu(Q_1(\nu))}{w(Q_1(\nu))^{t/p}}\lesssim w(Q_1(\nu))^{-t/q}.$$
On the other hand, Lemma \ref{esti} yields that there exists $C>0$ such that for any $\nu\in\Z^{2n}$,
\begin{equation}\label{nu0}
C^{-|\nu|}w(Q_1(0))\leq w(Q_1(\nu))\leq C^{|\nu|}w(Q_1(0)).
\end{equation}
Consequently, \eqref{upp} implies that
\begin{align*}
S^{t,\alpha,\beta}_{\mu}f_k(z)^t
&\lesssim\sum_{\nu\in\Z^{2n}}e^{-\frac{\beta}{4}|\nu|^2}w(Q_1(\nu))^{-t/q}
    \left(\int_{B_{2n}(\nu)}|f_k(u)|^pe^{-\frac{p\alpha}{2}|u|^2}w(u)dv(u)\right)^{t/p}\\
&\lesssim\sum_{\nu\in\Z^{2n}}e^{-\frac{\beta}{4}|\nu|^2}C^{\frac{t}{q}|\nu|}
    \left(\int_{B_{2n}(\nu)}|f_k(u)|^pe^{-\frac{p\alpha}{2}|u|^2}w(u)dv(u)\right)^{t/p}.
\end{align*}
Fix $\epsilon>0$. Then we can choose $R_1>0$ such that
$$\sum_{\nu\in\Z^{2n}\setminus B_{R_1}(0)}e^{-\frac{\beta}{4}|\nu|^2}C^{\frac{t}{q}|\nu|}<\epsilon^t.$$
Since $f_k\to0$ uniformly on compact subsets of $\C^n$, there exists $K_1>0$ such that for any $u\in \bigcup_{\nu\in\Z^{2n}\cap B_{R_1}(0)}B_{2n}(\nu)$,
$$|f_k(u)|<\epsilon$$
whenever $k>K_1$. Therefore, for any $k>K_1$ and $z\in B_{R_0}(0)$, Remark \ref{remark} together with the estimate \eqref{nu0} yields that
\begin{align*}
S^{t,\alpha,\beta}_{\mu}f_k(z)^t
&\lesssim\epsilon^t\sum_{\nu\in\Z^{2n}\cap B_{R_1}(0)}e^{-\frac{\beta}{4}|\nu|^2}C^{\frac{t}{q}|\nu|}w(B_{2n}(\nu))^{t/p}\\
&\qquad    +\|f_k\|^t_{F^p_{\alpha,w}}\sum_{\nu\in\Z^{2n}\setminus B_{R_1}(0)}e^{-\frac{\beta}{4}|\nu|^2}C^{\frac{t}{q}|\nu|}\\
&\lesssim\epsilon^t\sum_{\nu\in\Z^{2n}\cap B_{R_1}(0)}e^{-\frac{\beta}{4}|\nu|^2}C^{\frac{t}{q}|\nu|}w(Q_1(\nu))^{t/p}+\epsilon^t\\
&\lesssim\epsilon^t\sum_{\nu\in\Z^{2n}\cap B_{R_1}(0)}e^{-\frac{\beta}{4}|\nu|^2}C^{\frac{t}{q}|\nu|+\frac{t}{p}|\nu|}+\epsilon^t\\
&\lesssim\epsilon^t.
\end{align*}
Hence we conclude that $S^{t,\alpha,\beta}_{\mu}f_k\to0$ uniformly on $B_{R_0}(0)$.

{\bf Case 2: $p>q$.} By the proof of Theorem \ref{bdd>},
$$\sum_{\nu\in\Z^{2n}}\left(\frac{\mu(Q_1(\nu))^{\frac{q}{t}}}{w(Q_1(\nu))^{\frac{q}{p}-1}}\right)^{\frac{p}{p-q}}<\infty.$$
Consequently, if $q\leq t$, then by \eqref{upp}, H\"{o}lder's inequality and \eqref{nu0}, we obtain that
\begin{align*}
&S^{t,\alpha,\beta}_{\mu}f_k(z)^q\\
&\ \lesssim\left(\sum_{\nu\in\Z^{2n}}e^{-\frac{\beta}{4}|\nu|^2}\frac{\mu(Q_1(\nu))}{w(Q_1(\nu))^{t/p}}
    \left(\int_{B_{2n}(\nu)}|f_k(u)|^pe^{-\frac{p\alpha}{2}|u|^2}w(u)dv(u)\right)^{t/p}\right)^{q/t}\\
&\ \leq\sum_{\nu\in\Z^{2n}}e^{-\frac{q\beta}{4t}|\nu|^2}\frac{\mu(Q_1(\nu))^{q/t}}{w(Q_1(\nu))^{q/p}}
    \left(\int_{B_{2n}(\nu)}|f_k(u)|^pe^{-\frac{p\alpha}{2}|u|^2}w(u)dv(u)\right)^{q/p}\\
&\ \lesssim\left(\sum_{\nu\in\Z^{2n}}e^{-\frac{p\beta}{4t}|\nu|^2}w(Q_1(\nu))^{-\frac{p}{q}}
    \int_{B_{2n}(\nu)}|f_k(u)|^pe^{-\frac{p\alpha}{2}|u|^2}w(u)dv(u)\right)^{q/p}\\
&\ \lesssim\left(\sum_{\nu\in\Z^{2n}}e^{-\frac{p\beta}{4t}|\nu|^2}C^{\frac{p}{q}|\nu|}
    \int_{B_{2n}(\nu)}|f_k(u)|^pe^{-\frac{p\alpha}{2}|u|^2}w(u)dv(u)\right)^{q/p}.
\end{align*}
If $q>t$, then similarly,
\begin{align*}
&S^{t,\alpha,\beta}_{\mu}f_k(z)^q\\
&\ \lesssim\left(\sum_{\nu\in\Z^{2n}}e^{-\frac{\beta}{4}|\nu|^2}\frac{\mu(Q_1(\nu))}{w(Q_1(\nu))^{t/p}}
    \left(\int_{B_{2n}(\nu)}|f_k(u)|^pe^{-\frac{p\alpha}{2}|u|^2}w(u)dv(u)\right)^{t/p}\right)^{q/t}\\
&\ \lesssim\sum_{\nu\in\Z^{2n}}e^{-\frac{\beta}{4}|\nu|^2}\frac{\mu(Q_1(\nu))^{q/t}}{w(Q_1(\nu))^{q/p}}
    \left(\int_{B_{2n}(\nu)}|f_k(u)|^pe^{-\frac{p\alpha}{2}|u|^2}w(u)dv(u)\right)^{q/p}\\
&\ \lesssim\left(\sum_{\nu\in\Z^{2n}}e^{-\frac{p\beta}{4q}|\nu|^2}C^{\frac{p}{q}|\nu|}
    \int_{B_{2n}(\nu)}|f_k(u)|^pe^{-\frac{p\alpha}{2}|u|^2}w(u)dv(u)\right)^{q/p}.
\end{align*}
Therefore, using the same method as in Case 1, we can deduce that for any $z\in B_{R_0}(0)$, $S^{t,\alpha,\beta}_{\mu}f_k(z)\lesssim\epsilon$ whenever $k$ is large enough, which finishes the proof.
\end{proof}

\begin{lemma}\label{Duren}
Let $w$ be a weight on $\C^n$ and $0<p<\infty$. If $\{\varphi_k\}\subset L^p(wdv)$ and $\varphi\in L^p(wdv)$ satisfy $\|\varphi_k\|_{L^p(wdv)}\to\|\varphi\|_{L^p(wdv)}$ and $\varphi_k(z)\to\varphi(z)$ for almost every $z\in\C^n$, then $\varphi_k\to\varphi$ in $L^p(wdv)$.
\end{lemma}
\begin{proof}
See the proof of \cite[p. 21, Lemma 1]{Du}.
\end{proof}

The following proposition establishes an equivalent description for the compactness of $S^{t,\alpha,\beta}_{\mu}:F^p_{\alpha,w}\to L^q(wdv)$.

\begin{proposition}\label{cpt-eq}
Let $0<p,q,t,\alpha,\beta<\infty$, $w\in A^{\res}_{\infty}$, and let $\mu$ be a positive Borel measure on $\C^n$. Suppose that  $S^{t,\alpha,\beta}_{\mu}:F^p_{\alpha,w}\to L^q(wdv)$ is bounded. Then the following statements are equivalent:
\begin{enumerate}
	\item [(i)] $S^{t,\alpha,\beta}_{\mu}:F^p_{\alpha,w}\to L^q(wdv)$ is compact;
	\item [(ii)] for any bounded sequence $\{f_k\}\subset F^p_{\alpha,w}$ that converges to $0$ uniformly on compact subsets of $\C^n$, we have
	$$\left\|S^{t,\alpha,\beta}_{\mu}f_k\right\|_{L^q(wdv)}\to0.$$
\end{enumerate}
\end{proposition}
\begin{proof}
Suppose first that (i) holds. Let $\{f_k\}$ be a bounded sequence in $F^p_{\alpha,w}$ that converges to $0$ uniformly on compact subsets of $\C^n$. For any subsequence $\{f_{k_j}\}\subset \{f_k\}$, the compactness of $S^{t,\alpha,\beta}_{\mu}:F^p_{\alpha,w}\to L^q(wdv)$ ensures that there exist a subsequence $\{f_{k_{j_l}}\}\subset \{f_{k_j}\}$ and a function $h\in L^q(wdv)$ such that $S^{t,\alpha,\beta}_{\mu}f_{k_{j_l}}\to h$ in $L^q(wdv)$. Consequently, there exists a subsequence of $\{f_{k_{j_l}}\}$, still denoted by $\{f_{k_{j_l}}\}$, such that $S^{t,\alpha,\beta}_{\mu}f_{k_{j_l}}(z)\to h(z)$ for almost every $z\in\C^n$. On the other hand, Lemma \ref{uni0} implies that $S^{t,\alpha,\beta}_{\mu}f_k\to 0$ uniformly on compact subsets of $\C^n$. Therefore, $h=0$ and $S^{t,\alpha,\beta}_{\mu}f_{k_{j_l}}\to0$ in $L^q(wdv)$. The arbitrariness of $\{f_{k_j}\}\subset \{f_k\}$ then gives that
$$\left\|S^{t,\alpha,\beta}_{\mu}f_k\right\|_{L^q(wdv)}\to0.$$

Suppose now that (ii) holds. Let $\{g_k\}\subset F^p_{\alpha,w}$ be a bounded sequence. Then by Lemma \ref{pointwise} and Montel's theorem, there are a subsequence of $\{g_k\}$, still denoted by $\{g_k\}$, and an entire function $g$ on $\C^n$ such that $g_k\to g$ uniformly on compact subsets of $\C^n$. By Fatou's lemma, we have $g\in F^p_{\alpha,w}$. Hence $\{g_k-g\}$ is a bounded sequence in $F^p_{\alpha,w}$ that converges to $0$ uniformly on compact subsets of $\C^n$. Then by (ii),
\begin{equation}\label{norm0}
\left\|S^{t,\alpha,\beta}_{\mu}(g_k-g)\right\|_{L^q(wdv)}\to0.
\end{equation}
Moreover, Lemma \ref{uni0} yields that
\begin{equation}\label{point0}
S^{t,\alpha,\beta}_{\mu}(g_k-g)(z)\to0, \quad \forall z\in\C^n.
\end{equation}
In the case $t\geq1$, by Minkowski’s inequality,
$$S^{t,\alpha,\beta}_{\mu}(g_k-g)(z)\geq\left|S^{t,\alpha,\beta}_{\mu}g_k(z)-S^{t,\alpha,\beta}_{\mu}g(z)\right|,\quad \forall z\in\C^n.$$
If $q\geq1$, then Minkowski’s inequality again yields that
$$\left\|S^{t,\alpha,\beta}_{\mu}(g_k-g)\right\|_{L^q(wdv)}\geq
\left|\left\|S^{t,\alpha,\beta}_{\mu}g_k\right\|_{L^q(wdv)}-\left\|S^{t,\alpha,\beta}_{\mu}g\right\|_{L^q(wdv)}\right|,$$
and if $q<1$, then
$$\left\|S^{t,\alpha,\beta}_{\mu}(g_k-g)\right\|^q_{L^q(wdv)}\geq
\left|\left\|S^{t,\alpha,\beta}_{\mu}g_k\right\|^q_{L^q(wdv)}-\left\|S^{t,\alpha,\beta}_{\mu}g\right\|^q_{L^q(wdv)}\right|.$$
In the case $t<1$, we have
$$S^{t,\alpha,\beta}_{\mu}(g_k-g)(z)^t\geq\left|S^{t,\alpha,\beta}_{\mu}g_k(z)^t-S^{t,\alpha,\beta}_{\mu}g(z)^t\right|,\quad \forall z\in\C^n.$$
If $q/t\geq1$, then
$$\left\|S^{t,\alpha,\beta}_{\mu}(g_k-g)\right\|^t_{L^q(wdv)}\geq
\left|\left\|S^{t,\alpha,\beta}_{\mu}g_k\right\|^t_{L^q(wdv)}-\left\|S^{t,\alpha,\beta}_{\mu}g\right\|^t_{L^q(wdv)}\right|,$$
and if $q/t<1$, then
$$\left\|S^{t,\alpha,\beta}_{\mu}(g_k-g)\right\|^q_{L^q(wdv)}\geq
\left|\left\|S^{t,\alpha,\beta}_{\mu}g_k\right\|^q_{L^q(wdv)}-\left\|S^{t,\alpha,\beta}_{\mu}g\right\|^q_{L^q(wdv)}\right|.$$
Combining these inequalities with \eqref{norm0} and \eqref{point0}, we conclude that
$$S^{t,\alpha,\beta}_{\mu}g_k(z)\to S^{t,\alpha,\beta}_{\mu}g(z),\quad \forall z\in\C^n,$$
and
$$\left\|S^{t,\alpha,\beta}_{\mu}g_k\right\|_{L^q(wdv)}\to\left\|S^{t,\alpha,\beta}_{\mu}g\right\|_{L^q(wdv)},$$
which, together with Lemma \ref{Duren}, implies that $S^{t,\alpha,\beta}_{\mu}g_k\to S^{t,\alpha,\beta}_{\mu}g$ in $L^q(wdv)$. Therefore, $S^{t,\alpha,\beta}_{\mu}:F^p_{\alpha,w}\to L^q(wdv)$ is compact.
\end{proof}

We are now ready to characterize the compactness of $S^{t,\alpha,\beta}_{\mu}:F^p_{\alpha,w}\to L^q(wdv)$.

\begin{theorem}\label{cpt<}
Let $0<p\leq q<\infty$, $0<t,\alpha,\beta<\infty$, $w\in A^{\res}_{\infty}$, and let $\mu$ be a positive Borel measure on $\C^n$, finite on compact subsets of $\C^n$. Then $S^{t,\alpha,\beta}_{\mu}:F^p_{\alpha,w}\to L^q(wdv)$ is compact if and only if $\lim_{|z|\to\infty}G_{\mu}(z)=0$.
\end{theorem}
\begin{proof}
Suppose first that $S^{t,\alpha,\beta}_{\mu}:F^p_{\alpha,w}\to L^q(wdv)$ is compact. For any $z\in\C^n$, define
$$f_z(u):=\frac{e^{\alpha\langle u,z\rangle-\frac{\alpha}{2}|z|^2}}{w(B_1(z))^{1/p}},\quad u\in\C^n.$$
Then $\|f_z\|_{F^p_{\alpha,w}}\lesssim1$ and $f_z\to0$ uniformly on compact subsets of $\C^n$ as $|z|\to\infty$. Consequently, by the estimate \eqref{bdd-nece} and Proposition \ref{cpt-eq},
$$G_{\mu}(z)\lesssim\left\|S^{t,\alpha,\beta}_{\mu}f_z\right\|_{L^q(wdv)}\to0$$
as $|z|\to\infty$.

Suppose next that $\lim_{|z|\to\infty}G_{\mu}(z)\to0$. Then for any $\epsilon>0$, we can use the estimate \eqref{in1} to find $R>0$ such that
$$\sup_{\nu\in\Z^{2n}\setminus B_R(0)}\frac{\mu(Q_1(\nu))^{\frac{1}{t}}}{w(Q_1(\nu))^{\frac{1}{p}-\frac{1}{q}}}<\epsilon.$$
Let $\{f_k\}$ be a bounded sequence in $F^p_{\alpha,w}$ that converges to $0$ uniformly on compact subsets of $\C^n$. Then there exists $K\geq1$ such that whenever $k>K$, $|f_k(z)|<\epsilon$ for all $z\in\bigcup_{\nu\in\Z^{2n}\cap B_{R}(0)}B_{2n}(\nu)$. Therefore, noting that $G_{\mu}$ is bounded on $\C^n$ since $\mu$ is finite on compact subsets of $\C^n$, we can apply Proposition \ref{suff}, Lemma \ref{dis1} and Remark \ref{remark} to establish that for $k>K$,
\begin{align*}
\left\|S^{t,\alpha,\beta}_{\mu}f_k\right\|^q_{L^q(wdv)}
&\lesssim\sum_{\nu\in\Z^{2n}}\frac{\mu(Q_1(\nu))^{\frac{q}{t}}}{w(Q_1(\nu))^{\frac{q}{p}-1}}
\left(\int_{B_{2n}(\nu)}|f_k(u)|^pe^{-\frac{p\alpha}{2}|u|^2}w(u)dv(u)\right)^{\frac{q}{p}}\\
&\lesssim\epsilon^q\sum_{\nu\in\Z^{2n}\cap B_R(0)}
    \left(\int_{B_{2n}(\nu)}e^{-\frac{p\alpha}{2}|u|^2}w(u)dv(u)\right)^{\frac{q}{p}}\\
&\qquad+\epsilon^q\sum_{\nu\in\Z^{2n}\setminus B_{R}(0)}
    \left(\int_{B_{2n}(\nu)}|f_k(u)|^pe^{-\frac{p\alpha}{2}|u|^2}w(u)dv(u)\right)^{\frac{q}{p}}\\
&\lesssim\epsilon^q\left(\int_{\C^n}e^{-\frac{p\alpha}{2}|u|^2}w(u)dv(u)\right)^{\frac{q}{p}}+\epsilon^q\|f_k\|^q_{F^p_{\alpha,w}}\\
&\lesssim\epsilon^q,
\end{align*}
which, combined with Proposition \ref{cpt-eq}, implies that $S^{t,\alpha,\beta}_{\mu}:F^p_{\alpha,w}\to L^q(wdv)$ is compact.
\end{proof}

\begin{theorem}\label{cpt>}
Let $0<q<p<\infty$, $0<t,\alpha,\beta<\infty$, $w\in A^{\res}_{\infty}$, and let $\mu$ be a positive Borel measure on $\C^n$. Then $S^{t,\alpha,\beta}_{\mu}:F^p_{\alpha,w}\to L^q(wdv)$ is compact if and only if $G_{\mu}\in L^{\frac{pq}{p-q}}(dv)$.
\end{theorem}
\begin{proof}
In view of Theorem \ref{bdd>}, it suffices to prove the sufficiency. Suppose that $G_{\mu}\in L^{\frac{pq}{p-q}}(dv)$, and let $\{f_k\}\subset F^p_{\alpha,w}$ be a bounded sequence that converges to $0$ uniformly on compact subsets of $\C^n$. For any $\epsilon>0$, we can apply Lemma \ref{dis2} to find $R>0$ such that
$$\sum_{\nu\in\Z^{2n}\setminus B_R(0)}\left(\frac{\mu(Q_1(\nu))^{\frac{1}{t}}}{w(Q_1(\nu))^{\frac{1}{p}-\frac{1}{q}}}\right)^{\frac{pq}{p-q}}
<\epsilon^{\frac{pq}{p-q}}.$$
Since $f_k\to0$ uniformly on compact subsets of $\C^n$, there exists $K\geq1$ such that for any $k>K$ and $z\in \bigcup_{\nu\in\Z^{2n}\cap B_R(0)}B_{2n}(\nu)$,
$$|f_k(z)|<\epsilon.$$
Then, using Proposition \ref{suff}, H\"{o}lder's inequality, Lemma \ref{dis2} and Remark \ref{remark}, we have that for $k>K$,
\begin{align*}
\left\|S^{t,\alpha,\beta}_{\mu}f_k\right\|^q_{L^q(wdv)}
&\lesssim\sum_{\nu\in\Z^{2n}}\frac{\mu(Q_1(\nu))^{\frac{q}{t}}}{w(Q_1(\nu))^{\frac{q}{p}-1}}
  \left(\int_{B_{2n}(\nu)}|f_k(u)|^pe^{-\frac{p\alpha}{2}|u|^2}w(u)dv(u)\right)^{\frac{q}{p}}\\
&\lesssim\epsilon^q\|G_{\mu}\|^q_{L^{\frac{pq}{p-q}}(dv)}
    \left(\sum_{\nu\in\Z^{2n}\cap B_R(0)}\int_{B_{2n}(\nu)}e^{-\frac{p\alpha}{2}|u|^2}w(u)dv(u)\right)^{\frac{q}{p}}\\
&\qquad+\epsilon^q\left(\sum_{\nu\in\Z^{2n}\setminus B_R(0)}
    \int_{B_{2n}(\nu)}|f_k(u)|^pe^{-\frac{p\alpha}{2}|u|^2}w(u)dv(u)\right)^{\frac{q}{p}}\\
&\lesssim\epsilon^q\|G_{\mu}\|^q_{L^{\frac{pq}{p-q}}(dv)}\left(\int_{\C^n}e^{-\frac{p\alpha}{2}|u|^2}w(u)dv(u)\right)^{\frac{q}{p}}
    +\epsilon^q\|f_k\|^q_{F^p_{\alpha,w}}\\
&\lesssim\epsilon^q.
\end{align*}
Therefore, $\|S^{t,\alpha,\beta}_{\mu}f_k\|_{L^q(wdv)}\to0$, which, in conjunction with Proposition \ref{cpt-eq}, implies the desired compactness.
\end{proof}

Putting Theorems \ref{bdd<}, \ref{bdd>}, \ref{cpt<} and \ref{cpt>} together, we establish Theorem \ref{main1}.

\section{Proof of Theorem \ref{main2}}\label{proof2}

In this section, we are going to prove Theorem \ref{main2}. We first focus on the boundedness. Recall that
$$H_{\mu}(z)=\frac{\mu(B_1(z))}{w(B_1(z))^{\frac{1}{p}-\frac{1}{q}}},\quad z\in\C^n.$$

\begin{theorem}\label{T-bdd}
Let $0<p,q,\alpha<\infty$, $w\in A^{\res}_{\infty}$, and let $\mu$ be a positive Borel measure on $\C^n$.
\begin{enumerate}
	\item [(1)] If $p\leq q$, then $T^{\alpha}_{\mu}:F^p_{\alpha,w}\to F^q_{\alpha,w}$ is bounded if and only if $H_{\mu}$ is bounded on $\C^n$. Moreover,
	$$\left\|T^{\alpha}_{\mu}\right\|_{F^p_{\alpha,w}\to F^q_{\alpha,w}}\asymp\sup_{z\in\C^n}H_{\mu}(z).$$
	\item [(2)] If $p>q$, then $T^{\alpha}_{\mu}:F^p_{\alpha,w}\to F^q_{\alpha,w}$ is bounded if and only if $H_{\mu}\in L^{\frac{pq}{p-q}}(dv)$. Moreover,
	$$\left\|T^{\alpha}_{\mu}\right\|_{F^p_{\alpha,w}\to F^q_{\alpha,w}}\asymp\|H_{\mu}\|_{L^{\frac{pq}{p-q}}(dv)}.$$
\end{enumerate}
\end{theorem}
\begin{proof}
The sufficiency follows from \eqref{bt}, Theorems \ref{bdd<} and \ref{bdd>}. We consider the necessity. Suppose that $T^{\alpha}_{\mu}:F^p_{\alpha,w}\to F^q_{\alpha,w}$ is bounded.

(1) For any $z\in\C^n$, consider the function
$$f_z(u):=\frac{e^{\alpha\langle u,z\rangle-\frac{\alpha}{2}|z|^2}}{w(B_1(z))^{1/p}},\quad u\in\C^n.$$
Then by Lemma \ref{test}, $f_z\in F^p_{\alpha,w}$ with $\|f_z\|_{F^p_{\alpha,w}}\asymp1$. Consequently, we may use Lemma \ref{pointwise} and the boundedness of $T^{\alpha}_{\mu}:F^p_{\alpha,w}\to F^q_{\alpha,w}$ to establish that
\begin{align}\label{T-nece}
\nonumber\frac{\left\|T^{\alpha}_{\mu}\right\|^q}{w(B_1(z))}
\nonumber&\gtrsim\frac{1}{w(B_1(z))}\int_{B_1(z)}\left|T^{\alpha}_{\mu}f_z(u)\right|^qe^{-\frac{q\alpha}{2}|u|^2}w(u)dv(u)\\
\nonumber&\gtrsim\left|T^{\alpha}_{\mu}f_z(z)\right|^qe^{-\frac{q\alpha}{2}|z|^2}\\
\nonumber&=\left|\int_{\C^n}f_z(\xi)e^{\alpha\langle z,\xi\rangle-\alpha|\xi|^2}d\mu(\xi)\right|^qe^{-\frac{q\alpha}{2}|z|^2}\\
\nonumber&=\frac{1}{w(B_1(z))^{\frac{q}{p}}}\left(\int_{\C^n}e^{-\alpha|z-\xi|^2}d\mu(\xi)\right)^q\\
&\gtrsim\frac{\mu(B_1(z))^q}{w(B_1(z))^{\frac{q}{p}}},
\end{align}
which gives the desired result.

(2) Consider the function
$$F_{\tau}(z):=\sum_{\nu\in\Z^{2n}}c_{\nu}r_{\nu}(\tau)f_{\nu},$$
where $c=\{c_{\nu}\}_{\nu\in\Z^{2n}}\in l^p(\Z^{2n})$, $\{r_{\nu}\}_{\nu\in \Z^{2n}}$ is a sequence of Rademacher functions on $[0,1]$ (see \cite[Appendix A]{Du}), and
$$f_{\nu}=\frac{K^{\alpha}_{\nu}}{\|K^{\alpha}_{\nu}\|_{F^p_{\alpha,w}}}.$$
Then by Lemma \ref{test}, for almost every $\tau\in[0,1]$, $\|F_{\tau}\|_{F^p_{\alpha,w}}\lesssim\|c\|_{l^p(\Z^{2n})}$. Since $T^{\alpha}_{\mu}:F^p_{\alpha,w}\to F^q_{\alpha,w}$ is bounded, we have
$$\int_0^1\left\|T^{\alpha}_{\mu}F_{\tau}\right\|^q_{F^q_{\alpha,w}}d\tau
\lesssim\left\|T^{\alpha}_{\mu}\right\|^q\cdot\|c\|^q_{l^p(\Z^{2n})}.$$
Using Fubini's theorem and Khinchine's inequality (see \cite[Appendix A]{Du}), we obtain that
\begin{align*}
\int_0^1\left\|T^{\alpha}_{\mu}F_{\tau}\right\|^q_{F^q_{\alpha,w}}d\tau
&=\int_{\C^n}\int_0^1\left|\sum_{\nu\in\Z^{2n}}c_{\nu}r_{\nu}(\tau)T^{\alpha}_{\mu}f_{\nu}(z)\right|^q
    d\tau e^{-\frac{q\alpha}{2}|z|^2}w(z)dv(z)\\
&\asymp\int_{\C^n}\left(\sum_{\nu\in\Z^{2n}}|c_{\nu}|^2\left|T^{\alpha}_{\mu}f_{\nu}(z)\right|^2\right)^{\frac{q}{2}}
    e^{-\frac{q\alpha}{2}|z|^2}w(z)dv(z)\\
&\geq\int_{\C^n}\left(\sum_{\nu\in\Z^{2n}}|c_{\nu}|^2\left|T^{\alpha}_{\mu}f_{\nu}(z)\right|^2\chi_{Q_1(\nu)}(z)\right)^{\frac{q}{2}}
    e^{-\frac{q\alpha}{2}|z|^2}w(z)dv(z)\\
&=\sum_{\nu\in\Z^{2n}}|c_{\nu}|^q\int_{Q_1(\nu)}\left|T^{\alpha}_{\mu}f_{\nu}(z)\right|^qe^{-\frac{q\alpha}{2}|z|^2}w(z)dv(z),
\end{align*}
which, combined with Lemmas \ref{pointwise}, \ref{test} and Remark \ref{remark}, implies that
\begin{align*}
\int_0^1\left\|T^{\alpha}_{\mu}F_{\tau}\right\|^q_{F^q_{\alpha,w}}d\tau
&\gtrsim\sum_{\nu\in\Z^{2n}}|c_{\nu}|^qw(Q_1(\nu))\left|T^{\alpha}_{\mu}f_{\nu}(\nu)\right|^qe^{-\frac{q\alpha}{2}|\nu|^2}\\
&=\sum_{\nu\in\Z^{2n}}|c_{\nu}|^q\frac{w(Q_1(\nu))}{\|K^{\alpha}_{\nu}\|^q_{F^p_{\alpha,w}}}
    \left(\int_{\C^n}e^{2\alpha\Re\langle u,\nu\rangle-\alpha|u|^2-\frac{\alpha}{2}|\nu|^2}d\mu(u)\right)^q\\
&\asymp\sum_{\nu\in\Z^{2n}}|c_{\nu}|^qw(Q_1(\nu))^{1-\frac{q}{p}}\left(\int_{\C^n}e^{-\alpha|u-\nu|^2}d\mu(u)\right)^q\\
&\gtrsim\sum_{\nu\in\Z^{2n}}|c_{\nu}|^q\cdot\frac{\mu(Q_1(\nu))^q}{w(Q_1(\nu))^{\frac{q}{p}-1}}.
\end{align*}
Therefore, we get that
$$\sum_{\nu\in\Z^{2n}}|c_{\nu}|^q\cdot\frac{\mu(Q_1(\nu))^q}{w(Q_1(\nu))^{\frac{q}{p}-1}}\lesssim
\left\|T^{\alpha}_{\mu}\right\|^q\cdot\|c\|^q_{l^p(\Z^{2n})}.$$
By the duality and Lemma \ref{dis2}, we conclude that $H_{\mu}\in L^{\frac{pq}{p-q}}(dv)$, and $$\|H_{\mu}\|_{L^{\frac{pq}{p-q}}(dv)}\lesssim\left\|T^{\alpha}_{\mu}\right\|.$$
The proof is complete.
\end{proof}

We now turn to the compactness. The following lemma is needed.

\begin{lemma}\label{Tuni0}
Let $0<p,q,\alpha<\infty$, $w\in A^{\res}_{\infty}$, and let $\mu$ be a positive Borel measure on $\C^n$. Suppose that $T^{\alpha}_{\mu}:F^p_{\alpha,w}\to F^q_{\alpha,w}$ is bounded. Then for any bounded sequence $\{f_k\}\subset F^p_{\alpha,w}$ that converges to $0$ uniformly on compact subsets of $\C^n$, we have $T^{\alpha}_{\mu}f_k\to0$ uniformly on compact subsets of $\C^n$.
\end{lemma}
\begin{proof}
Since $T^{\alpha}_{\mu}:F^p_{\alpha,w}\to F^q_{\alpha,w}$ is bounded, we deduce from Theorems \ref{T-bdd}, \ref{bdd<} and \ref{bdd>} that $S^{1,\alpha,\alpha}_{\mu}:F^p_{\alpha,w}\to L^q(wdv)$ is bounded. Noting that for any $f\in F^p_{\alpha,w}$ and $z\in\C^n$,
$$\left|T^{\alpha}_{\mu}f(z)\right|e^{-\frac{\alpha}{2}|z|^2}\leq S^{1,\alpha,\alpha}_{\mu}f(z),$$
the desired result follows from Lemma \ref{uni0}.
\end{proof}

Based on Lemma \ref{Tuni0}, we can apply Lemma \ref{pointwise} and Montel's theorem to obtain the following equivalent description for the compactness of $T^{\alpha}_{\mu}:F^p_{\alpha,w}\to F^q_{\alpha,w}$. The proof is similar to that of Proposition \ref{cpt-eq} and so is omitted.

\begin{proposition}\label{Tcpt-eq}
Let $0<p,q,\alpha<\infty$, $w\in A^{\res}_{\infty}$, and let $\mu$ be a positive Borel measure on $\C^n$. Suppose that $T^{\alpha}_{\mu}:F^p_{\alpha,w}\to F^q_{\alpha,w}$ is bounded. Then the following statements are equivalent:
\begin{enumerate}
	\item [(i)] $T^{\alpha}_{\mu}:F^p_{\alpha,w}\to F^q_{\alpha,w}$ is compact;
	\item [(ii)] for any bounded sequence $\{f_k\}\subset F^p_{\alpha,w}$ that converges to $0$ uniformly on compact subsets of $\C^n$, we have
	$$\left\|T^{\alpha}_{\mu}f_k\right\|_{F^q_{\alpha,w}}\to0.$$
\end{enumerate}
\end{proposition}

The following theorem characterizes the compactness of $T^{\alpha}_{\mu}:F^p_{\alpha,w}\to F^q_{\alpha,w}$.

\begin{theorem}\label{T-cpt}
Let $0<p,q,\alpha<\infty$, $w\in A^{\res}_{\infty}$, and let $\mu$ be a positive Borel measure on $\C^n$.
\begin{enumerate}
	\item [(1)] If $p\leq q$ and $\mu$ is finite on compact subsets of $\C^n$, then $T^{\alpha}_{\mu}:F^p_{\alpha,w}\to F^q_{\alpha,w}$ is compact if and only if $\lim_{|z|\to0}H_{\mu}(z)=0$.
	\item [(2)] If $p>q$, then $T^{\alpha}_{\mu}:F^p_{\alpha,w}\to F^q_{\alpha,w}$ is compact if and only if $H_{\mu}\in L^{\frac{pq}{p-q}}(dv)$.
\end{enumerate}
\end{theorem}
\begin{proof}
The sufficiency follows from Theorem \ref{main1}, Propositions \ref{cpt-eq} and \ref{Tcpt-eq} and the inequality \eqref{bt}. The necessity of (2)  follows from Theorem \ref{T-bdd}. Suppose finally that $p\leq q$ and $T^{\alpha}_{\mu}:F^p_{\alpha,w}\to F^q_{\alpha,w}$ is compact. Then by Proposition \ref{Tcpt-eq}, $\lim_{|z|\to\infty}\left\|T^{\alpha}_{\mu}f_z\right\|_{F^q_{\alpha,w}}=0$, where
$$f_z(u)=\frac{e^{\alpha\langle u,z\rangle-\frac{\alpha}{2}|z|^2}}{w(B_1(z))^{1/p}},\quad u\in\C^n.$$
Therefore, by \eqref{T-nece}, we get that $\lim_{|z|\to0}H_{\mu}(z)=0$.
\end{proof}

Combining Theorem \ref{T-bdd} with Theorem \ref{T-cpt} finishes the proof of Theorem \ref{main2}.

\section{Concluding remarks}\label{CR}

We end this paper by a product-type characterization for the Carleson measures of weighted Fock spaces $F^p_{\alpha,w}$. For $0<p,\alpha<\infty$ and a positive Borel measure $\mu$ on $\C^n$, let $L^p_{\alpha}(\mu)$ be the space of measurable functions $f$ on $\C^n$ such that
$$\|f\|^p_{L^p_{\alpha}(\mu)}:=\int_{\C^n}|f(z)|^pe^{-\frac{p\alpha}{2}|z|^2}d\mu(z)<\infty.$$
The measure $\mu$ is said to be a $q$-Carleson measure for $F^p_{\alpha,w}$ if the embedding $I_d:F^p_{\alpha,w}\to L^q_{\alpha}(\mu)$ is bounded. The following theorem characterizes the $q$-Carleson measures for weighted Fock spaces $F^p_{\alpha,w}$ with $w\in A^{\res}_{\infty}$, which was proved in \cite{CFP} in the setting of one complex variable. The case of several complex variables can be proved similarly.

\begin{theorem}\label{CM}
Let $0<p,q,\alpha<\infty$, $w\in A^{\res}_{\infty}$, and let $\mu$ be a positive Borel measure on $\C^n$.
\begin{enumerate}
	\item [(1)] In the case $p\leq q$, $\mu$ is a $q$-Carleson measure for $F^p_{\alpha,w}$ if and only if the function
	$$z\mapsto\frac{\mu(B_1(z))}{w(B_1(z))^{q/p}}$$
	is bounded on $\C^n$.
	\item [(2)] In the case $p>q$, $\mu$ is a $q$-Carleson measure for $F^p_{\alpha,w}$ if and only if the function
	$$z\mapsto\frac{\mu(B_1(z))}{w(B_1(z))}$$
	belongs to $L^{\frac{p}{p-q}}(wdv)$.
\end{enumerate}
\end{theorem}

By the above theorem, we know that whether $\mu$ is a $q$-Carleson measure for $F^p_{\alpha,w}$ only depends on the ratio $\lambda=q/p$. Hence we say that $\mu$ is a $(w,\lambda)$-Fock--Carleson measure if the embedding $I_d:F^p_{\alpha,w}\to L^q_{\alpha}(\mu)$ is bounded for some $p,q>0$ with $\lambda=q/p$. Using Theorem \ref{CM} and \cite[Theorem 2.4]{Ch24}, and arguing as in the proof of \cite[Theorem 1.1]{ZZT}, we can establish the following characterization for $(w,\lambda)$-Fock--Carleson measures by using products of functions in weighted Fock spaces $F^p_{\alpha,w}$. We leave the details to the interested reader.

\begin{theorem}
Let $w\in A^{\res}_{\infty}$ and $\mu$ be a positive Borel measure on $\C^n$. For any integer $k\geq1$ and $j=1,\dots,k$, let
$$0<p_j,q_j<\infty,\quad \lambda=\sum_{j=1}^{\infty}\frac{q_j}{p_j}.$$
Then $\mu$ is a $(w,\lambda)$-Fock--Carleson measure if and only if for some (or any) $0<\alpha_j<\infty$, $j=1,\dots,k$, there exists $C>0$ such that for any $f_j\in F^{p_j}_{\alpha_j,w}$, $j=1,\dots,k$,
$$\int_{\C^n}\prod_{j=1}^k|f_j(z)|^{q_j}e^{-\frac{q_j\alpha_j}{2}|z|^2}d\mu(z)\leq C\prod_{j=1}^k\|f_j\|^{q_j}_{F^{p_j}_{\alpha_j,w}}.$$
\end{theorem}

\medskip




\begin{thebibliography}{99}

\bibitem{BCI} W. Bauer, L. A. Coburn and J. Isralowitz,
\newblock{Heat flow, BMO, and the compactness of Toeplitz operators,}
\newblock J. Funct. Anal. 259 (2010), no. 1, 57--78.

\bibitem{CFP23} C. Cascante, J. F\`{a}brega and D. Pascuas,
\newblock{Small Hankel operators on generalized weighted Fock spaces,}
\newblock Proc. Amer. Math. Soc. 151 (2023), no. 11, 4827--4839.

\bibitem{CFP} C. Cascante, J. F\`{a}brega and J. \'{A}. Pel\'{a}ez,
\newblock{Littlewood--Paley formulas and Carleson measures for weighted Fock spaces induced by $A_{\infty}$-type weights,}
\newblock Potential Anal. 50 (2019), no. 2, 221--244.

\bibitem{Ch24} J. Chen,
\newblock{Composition operators on weighted Fock spaces induced by $A_{\infty}$-type weights,}
\newblock Ann. Funct. Anal. 15 (2024), no. 2, Paper No. 22, 24 pp.

\bibitem{CHW} J. Chen, B. He and M. Wang,
\newblock{Absolutely summing Carleson embeddings on weighted Fock spaces with $A_{\infty}$-type weights,}
\newblock J. Operator Theory, in press. 

\bibitem{CW24} J. Chen and M. Wang,
\newblock{Weighted norm inequalities, embedding theorems and integration operators on vector-valued Fock spaces,}
\newblock Math. Z. 307 (2024), no. 2, Paper No. 36, 30 pp.

\bibitem{Du} P. L. Duren,
\newblock{Theory of $H^p$ spaces,}
\newblock Academic Press, New York-London, 1970.

\bibitem{Fu} R. Fulsche,
\newblock{Essential positivity for Toeplitz operators on the Fock space,}
\newblock Integral Equations Operator Theory 96 (2024), no. 3, Paper No. 21, 10 pp.

\bibitem{HL} Z. Hu and X. Lv,
\newblock{Toeplitz operators from one Fock space to another,}
\newblock Integral Equations Operator Theory 70 (2011), no. 4, 541--559.

\bibitem{Is} J. Isralowitz,
\newblock{Invertible Toeplitz products, weighted norm inequalities, and $A_p$ weights,}
\newblock J. Operator Theory 71 (2014), no. 2, 381--410.

\bibitem{IZ} J. Isralowitz and K. Zhu,
\newblock{Toeplitz operators on the Fock space,}
\newblock Integral Equations Operator Theory 66 (2010), no. 4, 593--611.

\bibitem{Ka} N. J. Kalton,
\newblock{Convexity, type and the three space problem,}
\newblock Studia Math. 69 (1980/81), no. 3, 247--287.

\bibitem{Me} T. Mengestie,
\newblock{On Toeplitz operators between Fock spaces,}
\newblock Integral Equations Operator Theory 78 (2014), no. 2, 213--224.

\bibitem{PZ15} J. Pau and R. Zhao,
\newblock{Carleson measures and Toeplitz operators for weighted Bergman spaces on the unit ball,}
\newblock Michigan Math. J. 64 (2015), no. 4, 759--796.

\bibitem{WZ} Z. Wang and X. Zhao,
\newblock{Invertibility of Fock Toeplitz operators with positive symbols,}
\newblock J. Math. Anal. Appl. 435 (2016), no. 2, 1335--1351.

\bibitem{ZZT} L. Zhou, D. Zhao and X. Tang,
\newblock{Carleson measures and Berezin-type operators on Fock spaces,}
\newblock{Banach J. Math. Anal. 18 (2024), no. 2, Paper No. 20, 17 pp.}

\bibitem{Zh} K. Zhu,
\newblock{Analysis on Fock spaces,}
\newblock Graduate Texts in Mathematics, 263. Springer, New York, 2012.






\end{thebibliography}
\end{document}